\documentclass{aptmod}

\authornames{S. Bonaccorsi, C. A. Tudor}

\shorttitle{Dissipative stochastic evolution equations driven by
  general noise}

\usepackage{eufrak}

\def\bA{\mathbb A}
\def\bC{\mathbb C}
\def\bE{\mathbb E}
\def\bF{\mathbb F}
\def\bN{\mathbb N}
\def\bR{\mathbb R}
\def\bV{\mathbb V}
\def\bX{\mathbb X}
\def\bY{\mathbb Y}
\def\cA{\mathcal A}
\def\cF{\mathcal F}
\def\cX{\mathcal X}

\numberwithin{equation}{section}

\begin{document}

\title{Dissipative stochastic evolution equations driven by general
  Gaussian and non-Gaussian noise}

  \vskip0.5cm

  \authorone[Universit\`a di Trento, Dipartimento di Matematica, Via
  Sommarive 14, 38050 Povo (Trento), Italia]{S. Bonaccorsi}

  \emailone{stefano.bonaccorsi@unitn.it}

  \authortwo[Laboratoire Paul Painlev\'e, Universit\'e de Lille 1,
 F-59655 Villeneuve d'Ascq, France.]{C. A. Tudor}

  \emailtwo{tudor@math.univ-lille1.fr; Associate member: Samos, Centre d'Economie de La Sorbonne, Universit\'e de Paris 1 Panth\'eon-Sorbonne, rue de Tolbiac, 75013 Paris, France. }

\begin{abstract}
  We study a class of stochastic evolution equations with a dissipative
  forcing nonlinearity and additive noise. The noise is assumed to
  satisfy rather general assumptions about the form of the covariance
  function; our framework covers examples of Gaussian processes, like
  fractional and bifractional Brownian motion and also non Gaussian
  examples like the Hermite process. We give an application of our
  results to the study of the stochastic version of a common model of
  potential spread in a dendritic tree. Our investigation is specially
  motivated by possibility to introduce long-range dependence in time of
  the stochastic perturbation.
\end{abstract}



\section{Introduction}

In recent years, existence, uniqueness and further properties of
solutions to stochastic equations in Hilbert spaces under
dissipativity assumptions has been widely discussed in the literature
since similar equations play an important r\^ole in stochastic models
of population biology, physics and mathematical finance (among
others), compare the monograph \cite{dpz:Ergodicity} for a thorough
discussion.

In this paper, by using semigroup methods, we shall discuss existence
and uniqueness of mild solutions to a class of stochastic evolution
equations driven by a stochastic process $X$ which is not necessarily
Gaussian.
In the Hilbert space $\bX$ we consider the following equation
\begin{equation}
  \label{eq:1}
  \begin{aligned}
  {\rm d}u(t) &= \bA u(t) + F(u(t)) \,{\rm d}t + {\rm d}X(t)
  \\
  u(0) &= u_0,
  \end{aligned}
\end{equation}
where $\bA$ and $F$ satisfy some dissipativity condition on $\bX$ and
$X$ is a general $\bX$-valued process that satisfies some specific
condition on the covariance operator.

Problems of the form of Equation \eqref{eq:1} arise in the modeling of
certain problems in neurobiology. In particular, in Section
\ref{sec:biology} we shall analyze a model of diffusion for electric
activity in a neuronal network, recently introduced in \cite{CM07},
driven by a stochastic term that is not white in time and
space. Notice further that, motivated by this model, we are concerned
with assumptions on the drift term which are not covered by those in
\cite{dpz:Ergodicity}.

\begin{assumpt}
  \label{Assumption A}
  The operator $\bA : D(\bA) \subset \bX \to \bX$ is associated with a form
  $({\mathfrak a}, \bV)$ that is densely defined, coercive and
  continuous; by standard theory of Dirichlet forms, compare
  \cite{Ou05}, the operator $\bA$ generates a strongly continuous,
  analytic semigroup $({S}(t))_{t\geq 0}$ on the Hilbert space $\bX$
  that is uniformly exponentially stable: there exist $M\ge 1$ and
  $\omega > 0$ such that $\|S(t)\|_{L(\bX)} \le M e^{-\omega t}$ for
  all $t \ge 0$.
\end{assumpt}

In the application of Section \ref{sec:biology}, the operator $\bA$ is
not self-adjoint, as the corresponding form ${\mathfrak a}$ is not
symmetric; also, since $\bV$ is not compactly embedded in $\bX$, it is
easily seen that the semigroup generated by $\bA$ is not compact hence
it is not Hilbert-Schmidt.

\begin{assumpt}
  \label{Assumption F}
  $F$ is an $m$-dissipative mapping with $\bV \subset D(F)$ and $F:
  \bV \to \bX$ is continuous with polynomial growth.
\end{assumpt}

Let us introduce the class of noises that we are concerned with.  We
define the mean of a $\bX$ valued process $(X_{t})_{t\in [0,T]}$ by
$m_{X}: [0,T] \to \bX$, $m_{X}(t)=\mathbb{E}X_{t}$ and the covariance $C_{X}:
[0,T] ^{2} \to L_{1}(\bX)$ by
\begin{align*}
  \langle C_{X}(t,s)u, v\rangle _{\bX}= \mathbb{E}\left[ \langle X_{t}-m_{X} (t)
    v\rangle _{\bX} \langle X_{s} -m_{X}(s), u\rangle _{\bX} \right]
\end{align*}
for every $s,t \in [0,T]$ and for every $u,v \in \bX$.

Let $Q$ be a nuclear self-adjoint operator on $\bX$ ($Q\in L_{1}(\bX)$
and $Q = Q^{\star} > 0$). It is well-known that $Q$ admits a sequence
$(\lambda _{j})_{j\geq 1}$ of eigenvalues such that $0 < \lambda_{j}
\downarrow 0$ and $\sum_{j}\lambda _{j \geq 1} <\infty$. Moreover, the
eigenvectors $(e_{j})_{j\geq 1}$ of $Q$ form an orthonormal basis of $\bX$.

Let $(x(t))_{t\in [0,T]} $ be a centered square integrable
one-dimensional process with a given covariance $R$. We define its
infinite dimensional counterpart by
\begin{equation*}
  X_{t} =\sum_{j=1} \sqrt{\lambda _{j}} x_{j}(t) e_{j} \qquad t\in [0,T],
\end{equation*}
where $x_{j}$ are independent copies of $x$. It is trivial to see that
the above series is convergent in $L^{2}(\Omega ; \bX)$ for every fixed $t\in [0,T]$ and
\begin{equation*}
\mathbb{E} \Vert X_{t}\Vert ^{2} _{\bX} =(Tr Q) R(t,t).
\end{equation*}

\begin{rem}
  The process $X$ is a $\bX$-valued centered process with covariance
  $R(t,s)Q$.
\end{rem}

\begin{assumpt}
  \label{ass:h1}
  We will assume that the covariance of the process $X$ satisfies the
  following condition:
  \begin{equation}\label{h1}
    (s,t) \to \frac{\partial^{2}R}{\partial s \partial t} \in
    L^{1}([0,T] ^{2}).
  \end{equation}
\end{assumpt}

We will treat several examples of stochastic processes that satisfy
(\ref{h1}). The first two examples are Gaussian processes (fractional
and bifractional Brownian motion) while the third example is
non-Gaussian (the Hermite process).

\begin{ex} {\it The process $X$ is a fractional Brownian motion (fBm)
    with Hurst parameter $H > \frac{1}{2}$. } We recall that its
  covariance equals, for every $s, t \in [0,T]$
  \begin{align*}
    R(s,t) = \frac{1}{2}\left ( s^{2H}+t^{2H}-\left|
        s-t\right|^{2H}\right).
  \end{align*}
  In this case $\frac{\partial^2R}{\partial s \partial
    t}=2H(2H-1)\left \vert t-s\right \vert^{2H-2}$ in the sense of
  distributions. Since $R$ vanishes on the axes, we have for every $s,t\in [0,T]$
  \begin{align*}
    R(s,t) = \int_0^{t} ds_{1} \int_0^{s} ds_2
    \frac{\partial^2R}{\partial s_1 \partial s_2}.
  \end{align*}
\end{ex}

\begin{ex} {\it $X$ is a bifractional Brownian motion with $H \in
    (0,1), K \in (0,1]$ and $2HK> 1.$ } Recall the the bifractional
  Brownian motion $(B^{H,K}_{t})_{t \in [0,T]}$ is a centered Gaussian
  process, starting from zero, with covariance
  \begin{equation}
    \label{cov-bi}
    R^{H,K}(t,s) := R(t,s)= \frac{1}{2^{K}}\left( \left(
        t^{2H}+s^{2H}\right)^{K} -\vert t-s \vert ^{2HK}\right)
  \end{equation}
  with $H\in (0,1)$ and $K\in (0,1]$.  Note that, if $K=1$ then
  $B^{H,1}$ is a fractional Brownian motion with Hurst parameter $H\in
  (0,1)$.


\end{ex}

\begin{ex} {\it A non Gaussian example: the Hermite process: }The
  driving process is now a Hermite process with selsimilarity order
  $H\in (\frac{1}{2}, 1)$. This process appears as a limit in the
  so-called {\it Non Central Limit Theorem} (see \cite{DM} or
  \cite{Ta1}). 


  We will denote by $(Z_{t}^{(q,H)})_{t\in \lbrack 0,1]}$ the Hermite
  process \textit{with self-similarity parameter} $H\in \left(
    1/2,1\right) $. Here $q\geq 1$ is an integer. The Hermite process
  can be defined in two ways: as a multiple integral with respect to
  the standard Wiener process $(W_{t})_{t\in \lbrack 0,1]}$; or as a
  multiple integral with respect to a fractional Brownian motion with
  suitable Hurst parameter.  We adopt the first approach throughout
  the paper: compare Definition \ref{defHermite} below.

\end{ex}

\vskip 1\baselineskip

In Section \ref{sec:stoch_conv} we treat the stochastic convolution
process
\begin{equation}
  \label{conv}
  W_{\bA}(t) = \int_0^t S(t-s) \, {\rm d}X_{s}.
\end{equation}
It is the weak solution of the linear stochastic evolution equation
${\rm d}Y(t) = \bA Y(t) \, {\rm d}t + {\rm d}X(t)$.  Our aim is to
prove that is a well-defined $\bX$-valued, mean square continuous,
$\mathcal{F}_t$-adapted process. We strenghten Assumption (\ref{h1})
by imposing the following.

\begin{assumpt}
  \label{Assumption 1}
  Let $X$ be given in the form
  \begin{align*}
    X_{t} =\sum_{j\geq 1}\sqrt{\lambda _{j} }x_{j}(t) e_{j}
  \end{align*}
  where $\lambda _{j}$, $e_{j}$ and $x_j(t)$ have been defined
  above. Suppose that the covariance $R$ of the process $(X_{t})_{t\in
    [0,T]}$ satisfies the following condition:
  \begin{align*}
    \left| \frac{\partial R}{\partial{s}\partial t} (s,t) \right| \leq
    c_{1}\vert t-s\vert ^{2H-2}+ g(s,t)
  \end{align*}
  for every $s,t \in [0,T]$ where $\vert g(s,t)\vert \leq
  c_{2}(st)^{\beta}$ with $\beta \in (-1,0)$, $H\in (\frac{1}{2},1)$
  and $c_{1}, c_{2}$ are strictly positive constant.
\end{assumpt}


\begin{rem}
  The fractional Brownian motion and the Hermite process satisfy
  Assumption \ref{Assumption 1} with $g$ identically zero. In the case of the
  bifractional Brownian motion the second derivative of the covariance
  can be divided into two parts. Indeed
  \begin{equation*}
    g(u,v)= c_{1} \vert u-v\vert ^{2HK-2}+ c_{2} (u^{2H}+ v^{2H})
    ^{K-2} (uv) ^{2H-1}:= g_{1}(u,v)+ g_{2}(u,v).
  \end{equation*}
  The part containing $g_{1}$ can be treated similarly to the case of
  the fractional Brownian motion. For the second term, note that
  \begin{equation*}
    u^{2H} + v^{2H} \geq 2(uv)^{H} \mbox{ and } (u^{2H}+ v^{2H})
    ^{K-2}\leq 2^{K-2} (uv) ^{H(K-2)}.
  \end{equation*}
  So,
  \begin{equation*}
    \vert g_{2}(u, v) \vert \leq cst. (uv) ^{HK -1}.
  \end{equation*}
  In conclusion, Assumption \ref{Assumption 1} is satisfied with
  $\beta =HK-1 \in (-1, 0)$.
\end{rem}

Our first main result is the following theorem concerning the
regularity of the stochastic convolution process under the Assumptions
\ref{Assumption A} and \ref{Assumption 1}.

\begin{thm}\label{thm:stoch_conv}
  In the above framework, fix $\alpha \in (0,H)$. Let $W_{\bA}$ be
  given by (\ref{conv}). Then $W_{\bA}$ exists in $L^2([0,T] \times
  \Omega;\bX)$ and it is ${\mathcal F}_t$-adapted.

  For every $\gamma <\alpha$ and $\varepsilon < \alpha -\gamma$ it holds that
  \begin{equation*}
    W_{\bA} \in C^{\alpha -\gamma -\varepsilon} \left( [0,T]; D(-A)^{\gamma} \right);
  \end{equation*}
  in particular for any fixed $t\in [0,T]$ the random variable
  $W_{\bA} (t)$ belongs to $D(-A) ^{\gamma }$.
\end{thm}

\medskip

Now we consider the solution of the stochastic evolution equation
(\ref{eq:1}). We consider generalized mild solutions in the sense of
\cite[Section 5.5]{dpz:Ergodicity}: an $\bX$-valued continuous and
adapted process $u = \{u_t,\ t \ge 0\}$ is a mild solution of
(\ref{eq:1}) if it satisfies $\Pr$-a.s.\ the integral equation
\begin{equation}
  \label{eq:integral}
  u(t) = S(t)u_0 + \int_0^t S(t-s) F(u(s)) \, {\rm d}s + W_{\bA}(t).
\end{equation}

\begin{thm}\label{thm:ex-intro}
  In our setting, let $u_0 \in D(F)$ (resp. $u_0 \in \bX$). Then there
  exists a unique mild (resp.  generalized) solution
  \begin{equation*}
    u \in L^2_{\cF}(\Omega;C([0,T];\bX)) \cap L^2_{\cF}(\Omega;L^2([0,T];\bV))
  \end{equation*}
  to equation (\ref{eq:1}) which depends continuously on the
  initial condition:
  \begin{equation}
    \label{eq:dip_d_iniz}
    {\mathbb E} \left|u(t;u_0) - u(t;u_1)\right|^2_\bX \leq C
    \left|u_0 - u_1\right|^2_\bX.
  \end{equation}
\end{thm}

\begin{rem}
  Even for a Wiener perturbation, this result is not contained in the
  existing literature since we does not assume any dissipativity or
  generation property of $\bA$ on $\bV$, compare \cite[Hypothesis 5.4
  and 5.6]{dpz:Ergodicity}.
\end{rem}

With this result at hand, we can solve the model of a complete
neuronal network recently proposed in \cite{CM07}.  It is
well known that any single
neuron can be schematized as a collection of a \emph{dentritic tree}
that ends into a \emph{soma} at one end of \emph{an axon}, hence as a
\emph{tree} in the precise sense defined within the mathematical field
of \emph{graph theory}.  By introducing stochastic terms we can model
the chaotic fluctuations of synaptic activity and post-synaptic
elaboration of electronic potential. There is sufficient sperimental
evidence that, in order to capture the actual behaviour of the
neurobiological tissues, infinite dimensional, stochastic, nonlinear
reaction-diffusion models are needed.

Previous models used simplified version of the neuronal network or
just concentrate on single parts of the cell: compare
\cite{BoMa-IDAQP} for a thorough analysis of the FitzHugh Nagumo
system on a neuronal axon or \cite{BoMu} for the analysis of the
(passive) electric propagation in a dendritic tree in the subtreshold
regime. In our model, instead, we are based based on the deterministic
description of the whole neuronal network that has been recently
introduced in \cite{CM07}; therefore, we avoid to sacrify the
biological realism of the neuronal model and, further, we add a
manifold of different possible stochastic perturbations that can be
chosen as a model for the enviromental influence on the system. Notice
that already in \cite{BoMu} a fractional Brownian motion with Hurst
parameter $H > \frac12$ was chosen in order to model the (apparently
chaotical) perturbance acting on a neuronal network.  This choice is
not a premiere in neuroscience, since different considerations show
that real inputs may exhibit long-range dependence and
self-similarity: see for instance the contributions in~\cite[Part
II]{RanDin03}.


\section{Wiener Integrals with respect to Hilbert valued  Gaussian and
  non-Gaussian processes with covariance structure measure}

In this section we discuss the construction of a stochastic integral
with respect to the process $(X_t)_{t \in [0,T]}$, which is not
necessarily Gaussian.

Since our non-Gaussian examples will be given by stochastic processes
that can be expressed as multiple Wiener-It\^o integrals, we need to
briefly recall the basic facts related to their constructed and their
basic properties.

\subsection{Multiple stochastic integrals}

Let $(W_{t})_{t\in [0,T]}$ be a classical Wiener process on a standard
Wiener space $\left( \Omega, {\mathcal{F}}, {\mathbb P}\right)$. If
$f\in L^{2}([0,T]^{n})$ with $n\geq 1$ integer, we introduce the
multiple Wiener-It\^{o} integral of $f$ with respect to $W$. The basic
reference is the monograph \cite{N}.  Let $f\in {\mathcal{S}_{m}}$ be
an elementary function with $m$ variables that can be written as
\begin{align*}
  f=\sum_{i_{1},\ldots ,i_{m}}c_{i_{1},\ldots ,i_{m}}1_{A_{i_{1}}\times
    \ldots \times A_{i_{m}}}
\end{align*}%
where the coefficients satisfy $c_{i_{1},\ldots i_{m}}=0$ if two
indices $i_{k}$ and $i_{l}$ are equal and the sets $A_{i}\in
{\mathcal{B}}([0,T])$ are disjoints. For such a step function $f$ we
define
\begin{align*}
  I_{m}(f)=\sum_{i_{1},\ldots ,i_{m}}c_{i_{1},\ldots ,i_{m}}
  W(A_{i_{1}}) \ldots W(A_{i_{m}})
\end{align*}%
where we put $W([a,b])=W_{b}-W_{a}$. It can be seen that the
application $I_{m}$ constructed above from ${\mathcal{S}}_{m}$ to
$L^{2}(\Omega )$ is an isometry on ${\mathcal{S}}_{m}$ , i.e.%
\begin{equation}
  \bE \left[ I_{n}(f)I_{m}(g)\right] =n! \langle f,g\rangle
  _{L^{2}([0,T]^{n})}\mbox{ if }m=n
  \label{isom}
\end{equation}%
and%
\begin{equation*}
  \bE \left[ I_{n}(f)I_{m}(g)\right] = 0 \mbox{ if }m\not=n.
\end{equation*}%

Since the set ${\mathcal{S}_{n}}$ is dense in $L^{2}([0,T]^{n})$ for
every $n\geq 1$, the mapping $I_{n}$ can be extended to an isometry
from $L^{2}([0,T]^{n})$ to $L^{2}(\Omega)$ and the above properties (\ref{isom})
hold true for this extension.

We recall the following hypercontractivity property  for the $L^{p}$ norm of a multiple
stochastic integral (see \cite[Theorem 4.1]{Major})
\begin{equation}
  \label{hyper}
  \bE \left| I_{m}(f) \right| ^{2m} \leq c_{m} \left( \bE I_{m}(f)^{2}
  \right)^{m}
\end{equation}
where $c_{m}$ is an explicit positive constant and $f\in
L^{2}([0,T]^{m})$.

\subsection{Wiener integrals: the one-dimensional case}
\label{3.3.2}

The idea to define Wiener integrals with respect to a centered
Gaussian (or non Gaussian) process $(X_{t})_{t\in [0,T]}$ is natural
and standard.  Denote by $R(t,s)= \mathbb{E}(X_{t}X_{s}) $ the covariance of
the process $X$.  Consider $\mathcal{E}$ the set of step functions on
$[0,T]$ defined as
\begin{equation}
\label{step} f= \sum_{i=0}^{n-1} c_{i} 1_{[t_{i} , t_{i+1}] }
\end{equation}
where $\pi: 0=t_{0} < t_{1} <\ldots < t_{n} =T$  denotes a partition
of $[0,T]$ and $c_{i}$ are real numbers. For a such $f$ it is
standard to define
\begin{equation*}
I(f)= \sum_{i=0}^{n-1} c_{i} \left( X_{t_{i+1}}-X_{t_{i}}\right).
\end{equation*}
It holds that
\begin{align*}
  \bE I(f)^{2} &= \sum_{i,j=0}^{n-1} c_{i}c_{j} \mathbb{E}\left(
    X_{t_{i+1}}-X_{t_{i}}\right)\left( X_{t_{i+1}}-X_{t_{i}}\right)
  \\
  &= \sum_{i,j=0}^{n-1} c_{i}c_{j}\left( R(t_{i+1}, t_{j+1})-
    R(t_{i+1}, t_{j})-R(t_{i}, t_{j+1}) +R(t_{i}, t_{j}) \right).
\end{align*}
The next step is to extend, by density, the application $I:
\mathcal{E} \to L^{2}(\Omega)$ to a bigger space, using the fact that
it is an isometry. This construction has been done in \cite{KRT} and
we will describe here the main ideas. In particular, we shall see that
the construction depends on the covariance structure of the process
$X$; the covariance of $X$ should define a measure on the Borel sets
of $[0,T]^{2}$. The function $R$ defines naturally a finite additive
measure $\mu$ on the algebra of finite disjoint rectangles included in
$[0,T]^{2}$ by
\begin{equation*}
  \mu (A) = R(b,d)+ R(a,c) -R(a,d)-R(c,b)
\end{equation*}
if $A = [a, b)\times [c,d)$.

In order to extend the Wiener integral to more general processes, we
 assume that the covariance of the process $X$ satisfies the
following condition:
\begin{equation}
  (s,t) \to \frac{\partial^{2}R}{\partial s \partial t} \in L^{1}([0,T] ^{2}).
\end{equation}
(compare with Assumption \ref{ass:h1}). This is a particular case of
the situation considered in \cite{KRT} where the integrator is assumed
to have a covariance structure measure in the sense that the
covariance $R$ defines a measure on $[0,T]^{2}$.

We have already seen some examples of stochastic processes that
satisfy (\ref{h1}): the {\em fractional Brownian motion with Hurst index bigger than $\frac{1}{2}$}, the {\em
  bifractional Brownian motion with $2HK>1$} and the {\em Hermite process}, for
instance.

The next step is to extend the definition of the Wiener integral to a
bigger class of integrands. We introduce $\left| \mathcal{H}\right|$
the set of measurable functions $f:[0,T] \to \mathbb{R}$ such that
\begin{equation}
  \label{normLu}
  \int_{0}^{T} \int_{0}^{T} \vert f(u)f(v)\vert \left|
    \frac{\partial^{2}R}{\partial u \partial v}(u,v) \right| \, {\rm
    d}u \, {\rm d}v <\infty.
\end{equation}
On the set $\left| \mathcal{H}\right|$ we define the inner product
\begin{equation}
  \label{scaLu}
  \langle f, h\rangle _{ \mathcal{H}} =\int_{0}^{T} \int_{0}^{T} f(u)
  h(v) \frac{\partial^{2}R}{\partial u\partial v}(u,v) \, {\rm d}u \,
  {\rm d}v
\end{equation}
and its associated seminorm
\begin{align*}
  \Vert f\Vert ^{2}_{\mathcal{H}} = \int_{0}^{T} \int_{0}^{T} f(u)
  f(v) \frac{\partial^{2}R}{\partial u\partial v}(u,v) \, {\rm d}u \,
  {\rm d}v.
\end{align*}
We also define
\begin{equation}
  \label{normlu-bar}
  \Vert f\Vert ^{2}_{\left| \mathcal{H}\right| }  =\int_{0}^{T}
  \int_{0}^{T} \vert f(u) f(v)\vert \left|
    \frac{\partial^{2}R}{\partial u\partial v}(u,v) \right| \, {\rm d}u \,
  {\rm d}v.
\end{equation}
It holds that $\mathcal{E}\subset \left| \mathcal{H}\right| $ and for
every $f,h\in \mathcal{E}$
\begin{equation}
  \label{iso18}
  \bE I(f)^{2} = \bE \Vert f \Vert^{2}_{\mathcal{H}}.
\end{equation}
The following result can be found in \cite{KRT}.

\begin{prop}
  The set $\mathcal{E}$ is dense in $\left| \mathcal{H}\right| $ with
  respect to $\left\| \cdot \right\|_{|\mathcal{H}|}$ and in
  particular to the seminorm $\left\| \cdot
  \right\|_{\mathcal{H}}$. The linear application $ \Phi: \mathcal{E}
  \longrightarrow L^2(\Omega)$ defined by
  \begin{align*}
    \varphi \longrightarrow I(\varphi)
  \end{align*}
  can be continuously extended to $\left| \mathcal{H}\right|$ equipped
  with the $\left\| \cdot \right\|_{\mathcal{H}}$-norm. Moreover we
  still have identity (\ref{iso18}) for any $\varphi \in \left|
    \mathcal{H}\right|$.
\end{prop}

We will set $\int_0^T \varphi \, {\rm d} X = \Phi(\varphi)$ and it
will be called {\bf the Wiener integral} of $\varphi$ with respect to
$X$.

\vskip0.3cm

We remark below that if the integrator process is a process in the
$n$th Wiener chaos then the Wiener integral with respect to $X$ is
again an element of the $n$th Wiener chaos.

\begin{rem}
  \label{rem:4.3}
  Suppose that the process $X$ can be written as
  $X_{t}=I_{k}(L_{t}(\cdot))$ with $k\geq 1$ and $L_{t} \in
  L^{2}([0,T]^{k})$ for every $t\in [0,T]$. Then for every $\varphi
  \in \left| \mathcal{H}\right|$ the Wiener integral $\int_0^T \varphi
  \, {\rm d} X$ is also in the $k$th Wiener chaos. Indeed, for simple
  functions of the form (\ref{step}) it is obvious and then we use the
  fact that the $k$th Wiener chaos is stable with respect to the
  $L^{2}$ convergence, that is, a sequence of random variables in the
  $k$th Wiener chaos convergent in $L^{2}$ has as limit a random
  variable in the $k$th Wiener chaos.
\end{rem}

\begin{rem}
  Assumption \ref{Assumption 1} implies condition (\ref{h1}). In
  particular the process $X$ whose covariance satisfies Assumption
  \ref{Assumption 1} have a covariance structure measure and the
  Wiener integral $\int_{0}^{T} \varphi \, {\rm d}X$ exists for every
  $\varphi \in \left| \mathcal{H}\right|$.  Indeed,
  \begin{multline*}
    \int_{0}^{T} \int_{0}^{T} \left| \frac{\partial
        R}{\partial{s}\partial t} (s,t) \right| \, {\rm d}s \, {\rm
      d}t \leq c_{1} \int_{0}^{T} \int_{0}^{T} \vert s-t\vert^{2H-2}
    \, {\rm d}s \, {\rm d}t + c_{2}\int_{0}^{T} \int_{0}^{T} (st) ^{\beta} \, {\rm
      d}s \, {\rm d}t
    \\
    \leq c \left( T^{2H} + T^{2(\beta +1)}\right).
  \end{multline*}
\end{rem}

Let us discuss now some examples. Firstly we refer to Gaussian
processes (fractional and bifractional Brownian motion).

\begin{ex}{\it The case of the fractional Brownian motion with
    $H>\frac{1}{2}$. } In this case $\left| \mathcal{H}\right|$ is the
  space of measurable functions $f:[0,T] \to \mathbb{R}$ such that
  \begin{equation*}
    \int_{0}^{T} \int_{0}^{T} \vert f(u) f(v) \vert u-v\vert ^{2H-2}
    \, {\rm d}u \, {\rm d}v <\infty .
  \end{equation*}
  On the other hand, for this integrator one can consider bigger
  classer of Wiener integrands. The natural space for the definition
  of the Wiener integral with respect to a Hermite process is the
  space $\mathcal{H}$ which is the closure of $\mathcal{E}$ with
  respect to the scalar product
  \begin{align*}
    \langle 1_{[0,t]}, 1_{[0,s]} \rangle _{\mathcal{H}} = R(t,s).
  \end{align*}
  We recall that $\mathcal{H}$ can be expressed using fractional
  integrals and it may contain distributions. Recall also that
  \begin{equation*}
    L^{2}([0,T])\subset L^{\frac{1}{H}} ([0,T]) \subset \left|
      \mathcal{H}\right| \subset \mathcal{H}
  \end{equation*}
  The Wiener integral with respect to Hermite processes can be also
  written as a Wiener integral with respect to the standard Brownian
  motion through a transfer operator (see e.g. \cite{N}).
\end{ex}

\begin{ex}
  \label{EBFBM} { \it The bifractional Brownian motion with $2HK>1$.}
  Recall the the bifractional
  Brownian motion $(B^{H,K}_{t})_{t \in [0,T]}$ is a centered Gaussian
  process, starting from zero, with covariance
  \begin{equation}
    \label{cov-bi}
    R^{H,K}(t,s) := R(t,s)= \frac{1}{2^{K}}\left( \left(
        t^{2H}+s^{2H}\right)^{K} -\vert t-s \vert ^{2HK}\right)
  \end{equation}
  with $H\in (0,1)$ and $K\in (0,1]$.

  We can write the covariance function as
  \begin{align*}
    R(s_1,s_2)=R_1(s_1,s_2) + R_2(s_1,s_2),
  \end{align*}
  where
  \begin{align*}
    &R_1(s_1,s_2) = \frac{1}{2^K} \left(s_1^{2H} + s_2^{2H}\right)^K -
    \left( s_1^{2HK} + s_2 ^{2HK}\right)
    \\
    \text{and }\qquad &R_2(s_1,s_2) =
    -\frac{1}{2^K} \vert s_2-s_1 \vert^{2HK} + s_1^{2HK} + s_2 ^{2HK}.
  \end{align*}
  We therefore have
  \begin{align*}
    \frac{\partial^2R_{1}}{\partial s_1 \partial s_2} =
    \frac{4H^2K(K-1)}{2^K} \left(s_1^{2H} + s_2^{2H}\right)^{K-2}
    s_1^{2H-1}s_2^{2H-1}.
  \end{align*}
  Since $R_1$ is of class $C^2((0,T]^2)$ and
  $\frac{\partial^2R_{1}}{\partial s_1 \partial s_2}$ is always
  negative, $R_1$ is the distribution function of a negative
  absolutely continuous finite measure, having
  $\frac{\partial^2R_{1}}{\partial s_1 \partial s_2}$ for density.

  Concerning the term $R_2$ we suppose $2HK > 1$. The part denoted by
  $R_2$ is (up to a constant) also the covariance function of a
  fractional Brownian motion of index $HK$ and
  $\frac{\partial^2R_{2}}{\partial s_1 \partial s_2} = 2HK
  (2HK-1)\left| s_1-s_2\right|^{2HK-2}$ which belongs of course to
  $L^1([0,T]^2)$. We also recall that the bifractional Brownian motion
  is a self-similar process with self-similarity index $HK$, it has
  not stationary increments, it is not Markovian and not a
  semimartingale for $2HK>1$.

  A significant subspace
  included in $\left| \mathcal{H}\right| $ is the set $L^2([0,T])$; if
  $K=1$ and $H = \frac{1}{2}$, there is even equality, since $X$ is a
  classical Brownian motion (see \cite{KRT}).
\end{ex}

Let us now give a non-Gaussian example.

\begin{ex} {\it The Hermite process $Z^{(q,H)}:=Z$ of order $q$
    with selfsimilarity order $H$. }

  The fractional Brownian process $(B_{t}^{H})_{t\in \lbrack 0,1]}$
  with Hurst parameter $H\in (0,1)$ can be written as
  \begin{equation*}
    B_{t}^{H}=\int_{0}^{t} K^{H}(t,s) \, {\rm d}W_{s},\quad t\in \lbrack 0,1]
  \end{equation*}%
  where $(W_{t},t\in \lbrack 0,T])$ is a standard Wiener process, the
  kernel $K^{H}\left( t,s\right)$ has the expression $c_{H}s^{1/2-H}
  \int_{s}^{t}(u-s)^{H-3/2}u^{H-1/2} \, {\rm d}u$ where $t>s$ and
  $c_{H}=\left( \frac{H(2H-1)}{\beta (2-2H,H-1/2)}\right) ^{1/2}$ and
  $\beta (\cdot ,\cdot )$ is the Beta function. For $t>s$, the
  kernel's derivative is $\frac{\partial K^{H}}{\partial
    t}(t,s)=c_{H}\left( \frac{s}{t}\right)
  ^{1/2-H}(t-s)^{H-3/2}$. Fortunately we will not need to use these
  expressions explicitly, since they will be involved below only in
  integrals whose expressions are known.

  We will denote by $(Z_{t}^{(q,H)})_{t\in \lbrack 0,1]}$ the Hermite
  process \textit{with self-similarity parameter} $H\in \left(
    1/2,1\right) $. Here $q\geq 1$ is an integer. Let us state the
  formal definition of this process.

  \begin{defn}
    \label{defHermite} The Hermite process $(Z_{t}^{(q,H)})_{t\in
      \lbrack 0,1]}$ of order $q\geq 1$ and with self-similarity
    parameter $H\in (\frac{1}{2},1)$ is given by
    \begin{multline}
      Z_{t}^{(q,H)} = d(H) \int_{0}^{t}\ldots\int_{0}^{t} {\rm
        d}W_{y_{1}} \ldots {\rm d}W_{y_{q}} \left(
        \int_{y_{1}\vee\ldots\vee y_{q}}^{t} \partial_{1}
        K^{H^{\prime}}(u,y_{1}) \ldots \partial_{1}
        K^{H^{\prime}}(u,y_{q}) {\rm d}u\right) ,
      \\
      t\in
      \lbrack0,1]
      \label{z1}
    \end{multline}
    where $K^{H^{\prime}}$ is the usual kernel of the fractional
    Brownian motion and
    \begin{equation}
      H^{\prime}=1+\frac{H-1}{q}\Longleftrightarrow(2H^{\prime}-2)q=2H-2.
      \label{H'}
    \end{equation}
  \end{defn}

  Of fundamental importance is the fact that the covariance of
  $Z^{\left( q,H\right) }$ is identical to that of fBm, namely
  \begin{equation*}
    \mathbb{E}\left[ Z_{s}^{\left( q,H\right) }Z_{t}^{\left( q,H\right) }\right]
    =\frac{1}{2}(t^{2H}+s^{2H}-|t-s|^{2H}).
  \end{equation*}
  The constant $d(H)$ is chosen to have the variance equal to 1. We
  stress that $Z^{\left( q,H\right) }$ is far from Gaussian for $q>1$,
  since it is formed of multiple Wiener integrals of order $q$ (see
  also \cite{Ta2}).

  \vskip0.3cm

  The basic properties of the Hermite process are listed below:

  \begin{itemize}
  \item the Hermite process $Z^{(q)}$ is $H$-self-similar and it has
    stationary increments.
  \item the mean square of the increment is given by
    \begin{align}
      \mathbb{E}\left[ \left\vert Z_{t}^{(q,H)}-Z_{s}^{(q,H)}\right\vert ^{2}
      \right] =|t-s|^{2H};
      \label{canonmetric}
    \end{align}
    as a consequence, it follows will little extra effort from
    Kolmogorov's continuity criterion that $Z^{(q,H)}$ has
    H\"{o}lder-continuous paths of any exponent $\delta <H$.
  \item it exhibits long-range dependence in the sense that
    \begin{align*}
      \sum_{n\geq 1} \mathbb{E}\left[
        Z_{1}^{(q,H)}(Z_{n+1}^{(q,H)}-Z_{n}^{(q,H)})\right] = \infty.
    \end{align*}
    In fact, the summand in this series is of order $n^{2H-2}$. This
    property is identical to that of fBm since the processes share the
    same covariance structure, and the property is well-known for fBm
    with $H>1/2$.
  \item for $q=1$, $Z^{(1,H)}$ is standard fBm with Hurst parameter
    $H$, while for $q\geq 2$ the Hermite process is not Gaussian. In
    the case $q=2$ this stochastic process is known as \emph{the
      Rosenblatt process}.
  \end{itemize}

  In this case the class of
  integrands $\mathcal{H}$ is the same as in the case of the
  fractional Brownian motion. We will also note that, from Remark
  \ref{rem:4.3}, the Wiener integral with respect to the Hermite
  process $\int_{0}^{T} \varphi \, {\rm d}Z$ is an element of the
  $k$th Wiener chaos. Moreover, it has been proven in \cite{MaTu} that
  for every $\varphi \in \left| \mathcal{H} \right|$ we have
  \begin{align*}
    \int_{0}^{T} f(u)dZ(u) = \int_{0}^{T}\ldots \int_{0}^{T}
    I(f)(y_{1}, y_{2}, \dots, y_{k}) \, {\rm d}B(y_{1}) \, {\rm
      d}B(y_{2}) \dots \, {\rm d}B(y_{k})
  \end{align*}
  where $(B_{t}) _{t\in [0,T]}$ is a Wiener process and we denoted by $I$ the following transfer operator
\begin{align*}
  I(f)(y_{1}, y_{2}, \dots, y_{k})= \int_{y_{1}\vee \dots\vee
    y_{k}}^{T} f(u) \partial_{1} K^{H'} (u, y_{1})...\partial_{1}
  K^{H'} (u, y_{k}) \, {\rm d}u
\end{align*}
where $H'$ is defined by (\ref{H'}).
\end{ex}

\subsection{The infinite-dimensional case}

Let
\begin{equation*}
  X_{t} =\sum_{j=1} \sqrt{\lambda _{j}} x_{j}(t) e_{j} \qquad t\in [0,T],
\end{equation*}
be a $\bX$-valued centered process with covariance $R(t,s)Q$.

Let $G:[0,T] \to L(\bX)$ and let $(e_{j}) _{j\geq 1}$ be a complete
orthonormal system in $\bX$. Assume that for every $j\geq 1$ the
function $G(\cdot) e_{j}$ belongs to the space $\left| \mathcal{H}
\right|$. The we define the Wiener integral of $G$ with respect to $X$
by
\begin{align*}
  \int_{0}^{T} G \, {\rm d}X = \sum_{j\ge 1} \sqrt{\lambda_{j}}
  \int_{0}^{T} G(s) e_{j} \, {\rm d}x_{j}(s)
\end{align*}
where the Wiener integral with respect to ${\rm d}x_{j}$ has been
defined above in  paragraph \ref{3.3.2}.

\begin{rem}
  The above integral is well-defined as an element of $L^{2}(\Omega;
  \bX)$ and we have the bound
  \begin{multline*}
    \bE \left\Vert \int_{0}^{T} G \, {\rm d}B \right\Vert^{2} \leq
    Tr(Q) \int_{0}^{T} \int_{0}^{T} \Vert G(u) \Vert_{L(V)} \Vert G(v)
    \Vert_{L(V)} \left| \frac{\partial ^{2} R}{\partial u\partial
        v}(u,v )\right| \, {\rm d}u \, {\rm d}v
    \\
    \leq Tr(Q) \left\| \Vert G(\cdot ) \Vert _{L(V)} \right\|_{\left|
        \mathcal{H} \right| } ^{2}.
  \end{multline*}
\end{rem}


\section{The stochastic convolution process}
\label{sec:stoch_conv}

There exists a well established theory on stochastic evolution
equations in infinite dimensional spaces, see Da Prato and Zabcyck
\cite{dpz:Ergodicity}, that we shall apply in order to show that
Eq.(\ref{eq:integral}) admits a unique solution. Let us recall from
Assumption \ref{Assumption A} that $\bA$ is the infinitesimal
generator of a strongly continuous semigroup $(S(t))_{t\geq 0}$, on
$\bX$ that is exponentially stable.

In this setting, we are concerned with the so-called stochastic
convolution process
\begin{equation*}
  \tag{\ref{conv}}
  W_{\bA}(t) = \int_0^t S(t-s) \, {\rm d}X_{s}.
\end{equation*}
It is the weak solution of the linear stochastic evolution equation
${\rm d}Y(t) = \bA Y(t) \, {\rm d}t + {\rm d}X(t)$.  Our aim is to
prove that is a well-defined mean square continuous,
$\mathcal{F}_t$-adapted process. Let us make the following assumption:

\begin{prop}
  Assume that the covariance function $R$ satisfies (\ref{h1}). Then,
  for every $t\in [0,T]$, the stochastic convolution  given by (\ref{conv})
  exists in $L^{2}([0,T] ;\bX)$ and it is $\mathcal{F} _{t}$ adapted.
\end{prop}

\begin{proof}
  We have that, by using the exponential stability of the semigroup
  $S(t)$ (see Assumption \ref{Assumption A})
  \begin{align*}
    \mathbb{E} & \left\| \int_{0}^{t} S(t-s) \, {\rm d}X(s)\right\|    _{\bX}^{2}
    \\
    &\leq Tr(Q) \int_{0}^{t} \int_{0} ^{t}\Vert S(t-u) \Vert _{L(\bX)}
    \Vert S(t-v) \Vert _{L(\bX)} \left| \frac{\partial ^{2}
        R}{\partial u\partial v}(u,v )\right| \, {\rm d}u \, {\rm d}v
    \\
    &\leq M^2 \, Tr(Q) \int_{0}^{T} \int_{0}^{T} e^{-\omega (t-u)}
    e^{-\omega (t-v)} \left| \frac{\partial ^{2} R}{\partial
        u\partial v}(u,v )\right| \, {\rm d}u \, {\rm d}v <\infty.
  \end{align*}
  The fact that  $W_{\bA}$ is adapted is obvious.
\end{proof}

\vskip0.3cm

The next step is to study the regularity (temporal and spatial) of the
stochastic convolution process. This will lead to a study of an
infinite sum of random variables with independent but not necessarily
Gaussian summands (they are elements in a fixed order Wiener chaos).   Let us recall the following result from
\cite[Theorem 3.5.1, page 76 and Theorem 2.2.1, page 32]{KwWo}.

\begin{prop}
  \label{lp}
  Let $\bX$ be a Hilbert space.
  \begin{enumerate}
  \item[a)] Let $p>4$ and $X_{1},\ldots , X_{n}$ be zero mean,
    independent $\bX$ valued random variables. Then
    \begin{multline*}
      \left( \bE \left\| \sum_{i=1}^{n} X_{i} \right\| _{\bX} ^{p}
      \right) ^{\frac{1}{p} }\leq c_{p} \left[ \left( \mathbb{E}\left\|
            \sum_{i=1}^{n} X_{i} \right\| _{\bX} ^{2} \right)
        ^{\frac{1}{2}} \right.
      \\
      \left. +\left( \mathbb{E}\left\| X_{n} \right\|_{\bX} ^{p} \vee \left(
            \mathbb{E}\left\| X_{n-1} \right\| _{\bX} ^{p} \vee \left( \ldots
              \vee \mathbb{E}\left\| X_{1} \right\| _{\bX} ^{p} \right) \right)
        \right) ^{\frac{1}{p}}\right].
    \end{multline*}
  \item[b)] Let $p>0$ and $X_{1}, X_{2}, \ldots , X_{n}, \dots$ be a
    sequence of independent $\bX$ valued random variables. If the
    series $\sum\limits_{i\geq 1} X_{i}$ converges almost surely to a
    random variable $S$ and for some $t>0$
    \begin{equation}
      \label{c100}
      \sum_{i\geq 1} \bE \left\| X_{i} \right\| _{\bX}^{p}
      1_{(\left\|X_{i}\right\| _{\bX} >t)} < \infty
    \end{equation}
    then $\bE \left\| S\right\| _{\bX}^{p} <\infty$ and $\displaystyle
    \bE \left\|S_{n} - S \right\| _{\bX}^{p} \stackrel{n\to
      \infty}{\longrightarrow} 0$ where $S_{n} = \sum\limits_{i=1}^{n}
    X_{i}$.
  \end{enumerate}
\end{prop}

\begin{rem}
  \label{rnou}
  It is not difficult to see that the point a) above implies that
  \begin{equation}
    \label{nou}
    \left( \bE \left\| \sum_{i=1}^{n} X_{i} \right\| _{\bX} ^{p}
    \right) ^{\frac{1}{p} }\leq c_{p} \left[ \left( \mathbb{E}\left\|
          \sum_{i=1}^{n} X_{i} \right\| _{\bX} ^{2} \right)
      ^{\frac{1}{2}} + \left( \sum_{i=1}^{n} \mathbb{E}\left\|
          X_{i}\right\| ^{p}\right) ^{\frac{1}{p}} \right].
  \end{equation}
\end{rem}

The following lemma is the main tool to get the regularity of the
stochastic convolution process $W_{\bA}$.

\begin{lem}
  \label{lem:4.5}
  Let $(X_{t}) _{t\in [0,T]} $ a stochastic process whose covariance
  $R$ satisfies Assumption \ref{Assumption 1}.  Denote, for every
  $\alpha \in (0,1)$,
  \begin{equation*}
    Y_{\alpha }(t)= \int_{0}^{t} (t-u) ^{-\alpha } S(t-u) \, {\rm
      d}X_{u}, \qquad t\in [0,T].
  \end{equation*}
  Then for every $\alpha \in (0,H)$, $Y_{\alpha}$ belongs to
  $L^{p}\left([0,T];\bX \right)$.
\end{lem}

\begin{proof}
  Suppose first that $X$ is Gaussian. Then, in order to show that
  $Y_{\alpha}$ is in $L^{p}\left([0,T]; \bX\right)$ it suffices to proves
  that it is in $L ^{2}\left( [0,T]; \bX\right)$. We have
  \begin{align*}
    \bE \left\| Y_{\alpha } (t) \right\| _{\bX}^{2} =& \bE \left\|
      \sum_{j\geq 1} \sqrt{\lambda _{j}} \int_{0}^{t} (t-u) ^{-\alpha
      } S(t-u)e_{j} \, {\rm d}x_{j} \right\|_{\bX}^{2}
    \\
    \leq& C\, \sum_{j\geq 1} \lambda _{j} \int_{0}^{t} \int_{0}^{t}
    (t-u) ^{-\alpha} (t-v)^{-\alpha} \left\| S(t-u) e_{j}
    \right\|_{\bX} \left\| S(t-v)e_{j} \right\|_{\bX}
    \\
    & \phantom{C\, \sum_{j\geq 1} \lambda _{j} \int_{0}^{t}
      \int_{0}^{t} (t-u) ^{-\alpha} (t-v)^{-\alpha} e^{-\omega _{1}
        (t-u) } e^{-\omega _{1} (t-v) }} \left| \frac{\partial
        ^{2}R}{\partial u \partial v} (u,v)\right| \, {\rm d}u \, {\rm
      d}v
    \\
    \leq& C \, ({\rm Tr} Q) \, \int_{0}^{t} \int_{0}^{t}
    (t-u)^{-\alpha} (t-v)^{-\alpha} e^{-\omega _{1} (t-u) } e^{-\omega
      _{1} (t-v) }\left| \frac{\partial ^{2}R}{\partial u \partial v}
      (u,v)\right| \, {\rm d}u \, {\rm d}v
    \\
    \leq& C ({\rm Tr} Q) \int_{0}^{t} \int_{0}^{t} (t-u)^{-\alpha}
    (t-v)^{-\alpha} e^{-\omega _{1} (t-u)} e^{-\omega _{1} (t-v)
    }\vert u-v \vert ^{2H-2} \, {\rm d}u \, {\rm d}v
    \\
    &+ C ({\rm Tr} Q) \int_{0}^{t} \int_{0}^{t} (t-u) ^{-\alpha }
    (t-v) ^{-\alpha } e^{-\omega _{1} (t-u) } e^{-\omega _{1} (t-v)
    }(uv) ^{\beta} \, {\rm d}u \, {\rm d}v
    \\
    :=& I_{1} + I_{2}.
  \end{align*}
  Concerning the term $I_{1}$, we can write
  \begin{multline*}
    I_{1}\leq 2C ({\rm Tr} Q)\int_{0}^{t} \int_{0}^{u} (t-u) ^{-\alpha
    } (t-v) ^{-\alpha } \vert u-v \vert ^{2H-2} \, {\rm d}v \, {\rm
      d}u
    \\
    \leq C ({\rm Tr} Q) \int_{0}^{t} u^{-2\alpha } u^{2H-1}
    \int_{0}^{1} z^{-\alpha } (1-z)^{2H-2} \, {\rm d}z \, {\rm d}u =
    C ({\rm Tr} Q) \, \int_{0}^{t} u^{-2\alpha } u^{2H-1} \, {\rm d}u
  \end{multline*}
  where we used the change of variable $\frac{v}{u}=z$. The last
  quantity is clearly finite if and only if $\alpha <H$. Concerning
  $I_{2}$ we have
  \begin{equation*}
    I_ {2} \leq \left( \int_{0}^{t} (t-u)^{-\alpha } e^{-\omega _{1}
        (t-u)} u^{\beta} \, {\rm d}u \right)^{2}
  \end{equation*}
  and this is always bounded by a constant (depending only on $T$)
  using the hypothesis imposed on $\alpha$ and $\beta$. We obtain thus
  the bound
  \begin{equation*}
    \mathbb{E}\left\| Y_{\alpha} (t) \right\| _{\bX}^{2}\leq C=C_{T}
  \end{equation*}
  for every $t\in [0,T]$.

  \vskip0.2cm

  Let us assume now that $X$ is not Gaussian and it belongs to the
  $k$-th Wiener chaos with $k\geq 2$. The process $X$ can be written as
  \begin{equation*}
    X_{t}= \sum_{j\geq 1} \sqrt{\lambda _{j} } e_{j} x_{j}(t)
  \end{equation*}
  where $x_{j}$ is an element of the $k$ th Wiener chaos with respect
  to the Wiener process $w_{j}$ and $(w_{j})_{j\geq 1}$ are
  independent real one-dimensional Wiener processes. Then
  \begin{equation*}
    Y_{\alpha }(t)= \sum_{j\geq 1} \sqrt{\lambda _{j} }  \int_{0}^{t}
    (t-u) ^{-\alpha } S(t-u)e_{j} \, {\rm d}x_{j}
  \end{equation*}
  is also an element in the $k$ th Wiener chaos (in the sense that
  every summand is in the $k$ th Wiener chaos with respect to
  $w_{j}$). Note also that, using the above computations from the
  Gaussian case we obtain
  \begin{equation*}
    \mathbb{E}\left\| Y_{\alpha }(t) \right\| ^{2} _{\bX} \leq C=C_{T}
  \end{equation*}
  for every $t \in [0,T]$. Denote by
  \begin{align*}
    S_{n, Y_{\alpha }} (t)= \sum_{j=1} ^{n} \sqrt{\lambda _{j} } e_{j}
    \int_{0}^{t} (t-u) ^{-\alpha } S(t-u) \, {\rm d}x_{j} :=
    \sum_{j=1}^{n} \sqrt{\lambda _{j}}A_{j}.
  \end{align*}
  By Proposition \ref{lp}, point a) and Remark \ref{rnou} we have
  \begin{align*}
    \left( \mathbb{E} \left\| S_{n, Y_{\alpha }} \right \| ^{p}
    \right) ^{\frac{1}{p}} \leq \left( \mathbb{E} \left\|
        \sum_{j=1}^{n} \sqrt{\lambda _{j}}A_{j}\right\| ^{2} _{\bX}
    \right) ^{\frac{1}{2}} + \left( \mathbb{E} \left\| \sqrt{\lambda
          _{n}}A_{n} \right\| _{\bX} ^{p} + \dots + \mathbb{E}\left\|
        \sqrt{ \lambda _{1} }A_{1} \right\| _{\bX} ^{p} \right)
    ^{\frac{1}{p}}.
  \end{align*}
  Using the hypercontractivity property of multiple stochastic
  integrals (\ref{hyper}), we get for every $i=1,.., n$
  \begin{multline*}
    \mathbb{E}\left\| \sqrt{\lambda _{i}} A_{i} \right\| _{\bX}^{p} =
    \lambda _{i} ^{\frac{p}{2}}\mathbb{E}\left| \int_{0}^{t} (t-u)
      ^{-\alpha } S(t-u) \, {\rm d}x_{j}\right| ^{p}
    \\
    \leq c_{p}\lambda _{i} ^{\frac{p}{2}} \left( \mathbb{E}\left|
        \int_{0}^{t} (t-u) ^{-\alpha } S(t-u) \, {\rm d}x_{j}
      \right|^{2} \right)^{\frac{p}{2}} \leq c_{p,T}
    \lambda_{i}^{\frac{p}{2}}.
  \end{multline*}
  As a consequence, since $0<\lambda _{i} \downarrow 0$
  \begin{align}
    \label{ineq1}
    \bE \left\| S_{n, Y_{\alpha }} (t)\right \|_{\bX}^{p} \leq c_{p,T}
    \left( \left( \sum_{i=1}^{n} \lambda _{i} \right)^{\frac{1}{2}} +
      \left( \sum_{i=1}^{n} \lambda _{i} ^{\frac{p}{2}}
      \right)^{\frac{1}{p}} \right)\leq c_{p,T}.
  \end{align}
  Now, since for every $t$ the sequence $S_{n, Y_{\alpha } }$ is
  convergent in $L^{2}(\Omega; \bX) $ as $n\to \infty$ we can find a
  sequence which converges almost surely. This subsequence will be
  again denoted by $S_{n, Y_{\alpha}}$. By Proposition \ref{lp} point
  b), since
  \begin{align*}
    \sum_{i\geq 1} \mathbb{E} \left\| \sqrt{ \lambda _{i} }A_{i}
    \right\| _{\bX} ^{p} 1_{( \sqrt{\lambda _{i} }A_{i} >t)} \leq \sum
    _{i\geq 1} \lambda _{i} ^{\frac{p}{2}} \mathbb{E}\left\| A_{i}
    \right\| _{\bX} ^{p} \leq c_{p,T}
\end{align*}
we obtain that for every $t$ the random variable $Y_{\alpha }(t)$ belongs to $L^{p} ([0,T] ; \bX)$ and \\
$\mathbb{E}\left\| S_{n, Y_{\alpha }}(t)- Y_{\alpha }(t) \right\| _{\bX} ^{p} \to _{n\to \infty }0$. Letting now $n \to \infty $ in (\ref{ineq1}) we obtain that
\begin{equation*}
\mathbb{E}\left\| Y_{\alpha }(t) \right\| _{\bX}^{p} \leq c_{p,T} .
\end{equation*}
and this finishes the proof.
\end{proof}

\begin{prop}
\label{pr:4.6}
  Suppose that $X$ satisfies Assumption \ref{Assumption 1} and fix
  $\alpha \in (0,H)$. Let $W_{\bA}$ be given by (\ref{conv}). Then for
  every $\gamma <\alpha $ and $\varepsilon < \alpha -\gamma $ it holds
  that
  \begin{equation*}
    W_{\bA} \in C^{\alpha -\gamma -\varepsilon} \left([0,T]; D((-\bA)^{\gamma})\right).
  \end{equation*}
  In particular for any fixed $t\in [0,T]$ the random variable
  $W_{\bA}(t)$ belongs to $D((-\bA)^{\gamma})$.
\end{prop}

\begin{proof}
  For $\alpha , \gamma \in (0,1)$, $p>1$ and $\psi \in L^{p} ([0,T] ,
  \bX)$ we define
  \begin{align*}
    R_{\alpha , \gamma } \psi(t) = \frac{ \sin(\alpha \pi )}{\pi }
    \int_{0}^{t} (t-u) ^{\alpha -1} (-A) ^{\gamma } S(t-u) \psi (u) \,
    {\rm d}u.
  \end{align*}
  Then, if $\alpha > \gamma + \frac{1}{p}$ it holds that
  \begin{align*}
    R_{\alpha, \gamma} \in L\left(L^{p}([0,T]; \bX);
      C^{\alpha-\gamma-\frac{1}{p}}([0,T]; D((-\bA)^{\gamma})) \right)
  \end{align*}

  It is standard to see that
  \begin{equation*}
    (-\bA)^{\gamma } X(t)= (R_{\alpha , \gamma} Y_{\alpha })(t)
  \end{equation*}
  where $Y_{\alpha }(t)= \int_{0}^{t} (t-u) ^{-\alpha } S(t-u) \, {\rm
    d}X(u)$. Since by the above lemma $Y_{\alpha }\in
  L^p([0,T];\bX)$ the conclusion follows.
\end{proof}

Next we will regard further properties of the stochastic convolution
process (\ref{conv}).  We are concerned with the $L^{p}$ norm of its supremum and
with its regularity with respect to the time variable. In the Gaussian
case the proofs basically follow the standard ideas from \cite[Chapter
5]{dpz:Stochastic}, while in the non-Gaussian case, the results are
new and they involve an analysis of the $L^{p}$ moments of multiple
Wiener-It\^o integrals.

\begin{lem}
  Assume that $(X_{t}) _{t\in [0,T] }$ satisfies Assumption
  \ref{Assumption 1} and let $W_{\bA}$ be given by (\ref{conv}). For
  any $p>\frac{1}{H}$, we have
  \begin{equation*}
    \bE \sup_{t\in [0,T]} \left\|W_{\bA}(t)\right\| _{\bX}^{p} \leq C.
  \end{equation*}
\end{lem}

\begin{proof}
  Note that for $t \in [0,T]$ and $0<\alpha <H$
  \begin{align*}
    W_{\bA}(t)= \frac{ \sin \pi \alpha }{\pi } \int_{0}^{t} S(t-s)
    (t-s)^{\alpha -1} Z_{\alpha }(s) \, {\rm d}s
  \end{align*}
  with $Z_{\alpha }(s) = \int_{0}^{s} S(s-u) (s-u)^{-\alpha} \, {\rm
    d}X_{u}$.  By H\"older's inequality with $p>\frac{1}{\alpha} >
  \frac{1}{H}$
  \begin{align*}
    \bE \sup_{t\in [0,T]} \left\| W_{\bA}(t) \right\| _{\bX}^{p} \leq
    \bE \int_{0}^{T} \left\| Z_{\alpha }(s)\right\| _{\bX}^{p} \, {\rm
      d}s.
  \end{align*}
  Now, since $0<\alpha <H$,
  \begin{align*}
    \bE \left\| Z_{\alpha }(s)\right\| _{\bX}^{2} \leq C \int_{0}^{T}
    e^{-\omega (t-u)} e^{-\omega (t-v)} \left| \frac{\partial
        ^{2}R}{\partial u\partial v} (u,v) \right| \, {\rm d}u \, {\rm
      d}v < C
  \end{align*}
  and by using the computations in the proof of Lemma \ref{lem:4.5}
   and the hypercontractivity property for
  multiple stochastic integrals (\ref{hyper}) we get $\bE \left\| Z_{\alpha}(s)
  \right\|_{bX}^{p} \leq C$.
\end{proof}

\vskip0.3cm

Let us now state our result concerning the regularity of $W_{\bA}$
with respect to the time variable.

\begin{prop}
  Fix $\alpha \in (0,H \wedge (\beta +1))$. Then the process
  $W_{\bA}(\cdot )$ has $\alpha$ H\"older continuous paths.
\end{prop}

\begin{proof}
  We will use Kolmogorov's continuity criterium for Hilbert valued
  stochastic processes (see \cite[Theorem 3.3]{dpz:Stochastic}). To
  this end, we need the evaluate the increment
  $W_{\bA}(t)-W_{\bA}(s)$. We can write
  \begin{align*}
    W_{\bA}(t)-W_{\bA}(s) =& \sum_{j\geq 1} \sqrt{\lambda _{j}}
    \int_{s}^{t} S(t-u) e_{j} \, {\rm d}x_{j}(u)
    \\
    &+ \sum_{j\geq 1} \sqrt{\lambda _{j}} \int_{0}^{s} (S(t-u)-S(s-u))
    e_{j} \, {\rm d}x_{j}(u)
    \\
    :=& I_{1}+ I_{2}.
  \end{align*}
  Concerning the first term, we get from Assumption \ref{Assumption 1}
  \begin{align*}
    \bE I_{1}^{2} \leq& (Tr Q) c_{1} \int_{s}^{t} \int_{s}^{t}
    e^{-\omega_{1}(t-u)} e^{-\omega_{1}(t-u)} \vert u-v\vert ^{2H-2}
    \, {\rm d}u \, {\rm d}v
    \\
    &+ (Tr Q) c_{2} \int_{s}^{t} \int_{s}^{t} e^{-\omega_{1}(t-u)}
    e^{-\omega_{1}(t-u)}(uv)^{\beta } \, {\rm d}u \, {\rm d}v
    \\
    \leq& C\left( \int_{s}^{t} \int_{s}^{t} \vert u-v\vert ^{2H-2} \,
      {\rm d}u \, {\rm d}v +\int_{s}^{t} \int_{s}^{t}(uv)^{\beta} \,
      {\rm d}u \, {\rm d}v \right)
    \\
    \leq& C (\vert t-s\vert ^{2H} + \vert t-s\vert ^{2(\beta +1)}).
  \end{align*}
  Following the proof of \cite[Theorem 5.1.3]{dpz:Stochastic}, we
  obtain $\bE I_{2}^{2} \leq C\vert t-s\vert ^{2\gamma}$ for any
  $\gamma \in (0,1)$. As a consequence
  \begin{align*}
    \bE \left\| W_{\bA}(t)-W_{\bA}(s) \right\| _{\bX}^{2} \leq C \left(
      \vert t-s\vert ^{2H} + \vert t-s\vert ^{2(\beta +1)} + \vert t-s
      \vert ^{2\gamma }\right)
  \end{align*}
  and by (\ref{hyper}) we will have that (as in the proof of Lemma \ref{lem:4.5}) for every $s$ close to $t$
  \begin{align*}
    \bE \left\| W_{\bA}(t)-W_{\bA}(s) \right\| _{\bX}^{p} \leq C_{p} \left(
      \vert t-s \vert^{pH} + \vert t-s \vert^{p(\beta +1)} \right)
  \end{align*}
  and this bound will imply the existence of an $\alpha$-H\"older continuous
  version of $W_{\bA}$.
\end{proof}

\begin{rem}
  In the case of the fractional Brownian motion and of the Rosenblatt
  process the order of continuity is $H$. For the bifractional
  Brownian motion, since $\beta +1=HK$, the stochastic convolution is
  $HK$ H\"older continuous.
\end{rem}


\section{Existence and uniqueness of the solution}
\label{sec:existence}

Let us first introduce the spaces where the solution will live.
\begin{defn}
  Let $L^2_{\cF}(\Omega; C([0,T]; \bX))$ denote the Banach space of all
  $\mathcal{F}_t$-measurable, pathwise continuous processes, taking
  values in $\bX$, endowed with the norm
  \begin{equation*}
    \left\|X\right\|_{L^2_{\cF}(\Omega;C([0,T];\bX))}=\left({\mathbb E} \sup_{t\in
      [0,T]}\left\|X(t)\right\| _{\bX}^2 \right)^{1/2}
  \end{equation*}
  while $L^2_{\cF}(\Omega;L^2([0,T];\bV))$ denotes the Banach space of all
  mappings $X: [0,T] \to \bV$ such that $X(t)$ is
  $\mathcal{F}_t$-measurable, endowed with the norm
  \begin{equation*}
    \left\|X\right\|_{L^2_{\cF}(\Omega;L^2([0,T];\bV))}=\left({\mathbb E} \int_0^T
      \left\|X(t)\right\|_{\bV}^2 \, {\rm d}t \right)^{1/2}.
  \end{equation*}
\end{defn}

We are concerned with Eq.~(\ref{eq:1}) that we mean to solve in {\em
  mild form}: a process $u \in {L^2_{\cF}(\Omega;C([0,T];\bX))} \cap
{L^2_{\cF}(\Omega;L^2([0,T];\bV))}$ is a solution to  Eq.~(\ref{eq:1}) if
it satisfies ${\mathbb P}$-a.s. the integral equation
\begin{align*}
  u(t) = S(t) u_0 + \int_0^t S(t-\sigma) F(u(\sigma)) \, {\rm d}\sigma
  + W_A(t), \qquad t \in [0,T].
\end{align*}

The strategy of the proof is classical, compare \cite[Theorem
5.5.8]{dpz:Ergodicity}: we consider the difference $(u(t) - W_A(t))_{t
  \in [0,T]}$ and we prove that it satisfies the mild equation and it
belongs to the relevant spaces.

\subsection{Existence of the  solution for deterministic equations}

Let us consider the following evolution equation
\begin{equation}
  \label{eq:d}
  \begin{aligned}
  \frac{\rm d}{{\rm d}t}y(t) &= \bA y(t) + F(z(t) + y(t))
  \\
  y(0) &= u_0,
  \end{aligned}
\end{equation}
where $A$ and $F$ satisfy the dissipativity condition on $\bX$ stated
in Assumptions \ref{Assumption A} and \ref{Assumption F}  and
$z$ is a trajectory of the stochastic convolution process, which
satisfies the regularity conditions stated in Theorem \ref{thm:stoch_conv},
\begin{align*}
  z \in C^{\alpha -\gamma -\varepsilon} \left( [0,T]; D(-A)^{\gamma}
  \right).
\end{align*}
The construction in this section is based on the techniques of
\cite[Section 5.5]{dpz:Ergodicity}; notice however that we are
concerned with a different kind of stochastic convolution and we do
not impose any dissipativity on the operators $\bA$ and $F$ on the
space $\bV$.

\begin{rem}
  The key point in the following construction is the observation that
  $\bV = D((-A)^{1/2})$, compare Remark \ref{rem:V=D(A^1/2)}. Further, in this
  case we impose the following bound: $\frac12 < H$. Therefore, we can
  and do assume that
  \begin{align*}
    z \in C^{H-1/2-\varepsilon} \left( [0,T]; D(-A)^{1/2} \right)
  \end{align*}
  for arbitrary $\varepsilon > 0$.

  Now, notice that the assumption on $F$ implies that $F:\bV \to \bX$
  is continuous, hence the process $(F(z(t)))_{t \in [0,T]}$ is
  continuous and satisfies $\sup_{t \in [0,T]} ||F(z(t))|| _{\bX} <
  +\infty$.
\end{rem}

Let us introduce the Yosida approximations $F_\alpha$ of $F$. It is
known that $F_\alpha$ are Lipschitz continuous, dissipative mappings
such that, for all $u \in \bV$, it holds $F_\alpha(u) \to F(u)$ in
$\bX$, as $\alpha \to 0$.

In this part, we are concerned with the following approximation of
Eq.~(\ref{eq:d}):
\begin{equation}
  \label{eq:d_app}
  \begin{aligned}
    \frac{\rm d}{{\rm d}t} y_\alpha(t) &= \bA y_\alpha(t) +
    F_\alpha(z(t) + y_\alpha(t))
    \\
    y_\alpha(0) &= u_0.
  \end{aligned}
\end{equation}

\begin{lem}\label{prop:sol_approx}
  Let $x\in \bX$. Then, for any $\alpha >0$ there exists a unique mild
  solution $y_{\alpha}(t,x)$ to Eq.~(\ref{eq:d_app}) such that
  \begin{align*}
    y_{\alpha} \in C([0,T];\bX) \cap L^2([0,T];\bV).
  \end{align*}
\end{lem}

\begin{proof}
  Since $F_\alpha$ are Lipschitz continuous, the existence of the
  solution to (\ref{eq:d_app}) is standard. It remains to prove the
  existence of a estimate that is uniform in $\alpha$.

  By the assumptions on $\bA$ there exists $\omega > 0$ such that
  $\langle \bA u,u \rangle \le -\omega \|u\|_{\bV}^2$, compare also
  Remark \ref{rem:coercive}; using the dissipativity of $F$ we have
  \begin{align*}
    \frac12 ||y_\alpha(t)||_{\bX}^2 &= \frac12 ||u_0||_{\bX}^2 + \int_0^t \langle \bA
    y_\alpha(s), y_\alpha(s) \rangle_{\bX} \, {\rm d}s + \int_0^t
    \langle F_\alpha(z(s) + y_\alpha(s)), y_\alpha(s) \rangle_{\bX} \,
    {\rm d}s
    \\
    &\le \frac12 ||u_0||_{\bX}^2 - \omega \int_0^t \| y_\alpha(s)\|_{\bV}^2 \,
    {\rm d}s + \int_0^t \langle F_\alpha(z(s)), y_\alpha(s)
    \rangle_{\bX} \, {\rm d}s
    \\
    &\le \frac12 ||u_0||_{\bX} ^2 - \omega \int_0^t \| y_\alpha(s)\|_{\bV}^2 \,
    {\rm d}s + T \sup_{t \in [0,T]} ||F(z(t))||^2_{\bX} + \int_0^t
    ||y_\alpha(s)|| _{\bX}^2 \, {\rm d}s
  \end{align*}
  which implies, by an application of Gronwall's lemma, that
  \begin{equation}
    \label{eq:2}
    \sup_{t \in [0,T]}\left( \frac12 ||y_\alpha(t)||_{\bX}^2 + \omega \int_0^t \|
    y_\alpha(s)\|_{\bV}^2 \, {\rm d}s \right) \le C(T,u_0,z).
  \end{equation}
  Notice that the constant on the right-hand side is independent of
  $\alpha$.
\end{proof}

\begin{lem}
  For every $\alpha > 0$, $u_0, u_1 \in \bX$, it holds
  \begin{equation}
    \label{eq:reg_dip_in_cond_app}
    \sup_{t \in [0,T]} ||y^{u_0}_\alpha(t) - y^{u_1}_\alpha(t)||^2_{\bX}
    \le C ||u_0 - u_1|| _{\bX}
    ^2.
  \end{equation}
\end{lem}

\begin{proof}
   us consider the difference $y_{\alpha}^{u_0}(t) -
  y_{\alpha}^{u_1}(t)$, for $x, \bar{x} \in H$:
  \begin{align*}
    \frac{\rm d}{{\rm d}t} \left[ y_{\alpha}^{u_0}(t) -
      y_{\alpha}^{u_1}(t) \right] = \bA \left[y_{\alpha}^{u_0}(t) -
      y_{\alpha}^{u_1}(t)\right] + \left[F_{\alpha}(z(t) +
      y_{\alpha}^{u_0}(t)) - F_{\alpha}(z(t) +
      y_{\alpha}^{u_1}(t))\right]
  \end{align*}
  hence
  \begin{multline*}
    ||y_{\alpha}^{u_0}(t) - y_{\alpha}^{u_1}(t)||_{\bX}^2
   = \left\|u_0 - u_1 \right\|_{\bX}^2 + 2 \int_0^t \left\langle A (
      y_{\alpha}^{u_0}(s) - y_{\alpha}^{u_1}(s)), y_{\alpha}^{u_0}(s)
      - y_{\alpha}^{u_1}(s) \right\rangle \, {\rm d}s
    \\
    + 2\int_0^t \left\langle F_{\alpha}(y_{\alpha}^{u_0}(s)) -
      F_{\alpha}(y_{\alpha}^{u_1}(s)), y_{\alpha}^{u_0}(s) -
      y_{\alpha}^{u_1}(s) \right\rangle \, {\rm d}s
  \end{multline*}
  and therefore
  \begin{align*}
    \left\| y_{\alpha}^{u_0}(t) - y_{\alpha}^{u_1}(t )\right\|^2_\bX
    \leq \left\| u_0 - u_1 \right\|^2_\bX - 2 \omega \int_0^t \left\|
      y_{\alpha}^{u_0}(s) - y_{\alpha}^{u_1}(s) \right\|^2_\bX \, {\rm
      d}s.
  \end{align*}
  Applying Gronwall's lemma we obtain
  \begin{equation}
    \label{eq:dip_dato_iniz}
    \left\| y_{\alpha}^{u_0}(t) - y_{\alpha}^{u_1}(t) \right\|^2_{\bX}
    \leq e^{-2 \omega t}\left\|u_0 - u_1 \right\|^2_{\bX}.
  \end{equation}
\end{proof}

\begin{lem}
  The sequence $(y_\alpha)_{\alpha > 0}$ is a Cauchy sequence in
  $C([0,T];\bX) \cap L^2([0,T];\bV)$.
\end{lem}

\begin{proof}
  Let $\alpha, \beta > 0$. Then we compute
  \begin{align*}
    \frac{\rm d}{{\rm d}t} \left[ y_{\alpha}(t) - y_{\beta}(t) \right]
    = \bA \left[y_{\alpha}(t) - y_{\beta}(t)\right] +
    \left[F_{\alpha}(z(t) + y_{\alpha}(t)) - F_{\beta}(z(t) +
      y_{\beta}(t))\right]
  \end{align*}
  Now, let us recall that
  \begin{align*}
    \langle F_\alpha(x) - F_\beta(y), x-y \rangle_{\bX} \le (\alpha +
    \beta) \left[ |F_\alpha(x)|^2 + |F_\beta(y)|^2 \right]
  \end{align*}
  for all $x, y \in \bV$, $\alpha, \beta > 0$ (compare
  \cite[Proposition 5.5.4]{dpz:Ergodicity}; it follows that
  \begin{multline}
    \label{eq:3}
    \frac12 \left\| y_{\alpha}(t) - y_{\beta}(t) \right\|^2_{\bX} +
    \omega \int_0^t \left\| y_{\alpha}(s) - y_{\beta}(s)
    \right\|^2_{\bX} \, {\rm d}s
    \\
    \le (\alpha + \beta) \int_0^t \left\| F_{\alpha}(z(s) +
      y_{\alpha}(s))\right\|^2_{\bX} + \left\| F_{\beta}(z(s) +
      y_{\beta}(s))\right\|^2_{\bX} \, {\rm d}s.
  \end{multline}
  Since $F: \bV \to \bX$ is continuous, it follows that for some
  contant $L > 0$
  \begin{multline*}
    \left\| F_{\alpha}(z(s) + y_{\alpha}(s))\right\|^2_{\bX} \le \left\|
      F(z(s) + y_{\alpha}(s))\right\|^2_{\bX}
    \\
    \le L \, \left\| z(s) + y_{\alpha}(s)\right\|^2_{\bV} \le 2L \,
    \left[ \|z(s)\|^2_{\bV} + \|y_{\alpha}(s)\|^2_{\bV} \right]
  \end{multline*}
  hence by using estimate (\ref{eq:2})
  \begin{multline*}
    \int_0^T \left\| F_{\alpha}(z(s) + y_{\alpha}(s))\right\|^2_{\bX} +
    \left\| F_{\beta}(z(s) + y_{\beta}(s))\right\|^2_{\bX} \, {\rm d}s
    \\
    \le 2LT \|z\|_{C([0,T];\bV)}^2 + C(T,u_0,z,\omega,L)
  \end{multline*}
  is bounded by a constant that does not depend on $\alpha$ and
  $\beta$. If we put the above estimate in (\ref{eq:3}) we obtain
  \begin{align*}
    \frac12 \sup_{t \in [0,T]} \left\| y_{\alpha}(t) - y_{\beta}(t)
    \right\|^2_{\bX} + \omega \int_0^T \left\| y_{\alpha}(s) -
      y_{\beta}(s) \right\|^2_{\bX} \, {\rm d}s \le C (\alpha + \beta)
  \end{align*}
  which easily implies the thesis.
\end{proof}

\begin{thm}
  For any $z \in C([0,T];\bV)$ there exists a unique solution
  $(y(t))_{t \in [0,T]}$ to Eq.~(\ref{eq:d}),
  \begin{align*}
    y \in C([0,T];\bV) \cap L^2([0,T];\bX)
  \end{align*}
  and it depends continuously on the initial condition $u_0 \in \bX$.
\end{thm}

\begin{proof}
  Since $y_\alpha$ is a Cauchy sequence in $C([0,T];\bV) \cap
  L^2([0,T];\bX)$ it converges to a unique function $y$ in the same
  space; it remains to show that $(y(t))_{t \in [0,T]}$ actually
  solves (\ref{eq:d}). Also, the continuous dependence on the initial
  condition follows from the same property proved for the
  approximating functions $y_\alpha$, since the estimate in
  (\ref{eq:reg_dip_in_cond_app}) does not depend on $\alpha$ and it is
  conserved at the limit.

  By the claimed convergence of $y_\alpha$, since $J_\alpha$ is a
  sequence of continuous mapping that converges to the identity, it
  holds that $J_\alpha(y_\alpha(s)) \to y(s) \in \bV$ a.s.\ on
  $[0,T]$. Therefore, by the continuity of $F$, it follows that
  \begin{align*}
    F_\alpha(z(s) + y_\alpha(s)) \to F(z(s) + y(s)) \in \bX \qquad
    \text{a.s.\ on $[0,T]$.}
  \end{align*}
  Now we use Vitali's theorem (the Uniform Integrability Convergence
  Theorem, compare \cite[Theorem 9.1.6]{Rosenthal}), to conclude that
  \begin{align*}
    \int_0^t S(t-s) F_\alpha(z(s) + y_\alpha(s)) \, {\rm d}s
    \longrightarrow \int_0^t S(t-s) F(z(s) + y(s)) \, {\rm d}s.
  \end{align*}
\end{proof}


\section{A network model for a neuronal cell}
\label{sec:biology}

In this paper we aim to investigate a mathematical model of a complete
neuron which is subject to stochastic perturbations; for a complete introduction to the
biological motivations, see \cite{KeSn}.
Our model is based on the deterministic one for the whole neuronal
network that has been recently introduced in \cite{CM07}; we shall
borrow from this paper the basic analytical framework for the
well-posedness of the problem.

We treat the neuron as a simple graph with different kind of
(stochastic) evolutions on the edges and dynamic Kirchhoff-type
condition on the central node (the soma). This approach is made
possible by the recent development of techniques of network evolution
equations; hence, as opposite to most of the papers in the literature,
which concentrate on some parts of the neuron, could it be the
dendritic network, the soma or the axon, we take into account the
complete cell.

In this paper, we schematize a neuron as a network by considering
\begin{itemize}
\item a FitzHugh-Nagumo (nonlinear) system on the axon, coupled with
\item a (linear) Rall model for the dendritical tree, complemented
  with
\item Kirchhoff-type rule in the soma.
\end{itemize}
It is commonly accepted that dendrites conduct electricity in a
passive way. The well known Rall's model \cite{Ra59, Ra60} simplify
the analysis of this part by considering a simpler, concentrated
``equivalent cylinder'' (of finite length $\ell_d$) that schematizes a
dendritical tree; he showed that a linear cable equation fits
experimental data on dendritical trees quite well, provided that it is
complemented by a suitable dynamical conditions imposed in the
interval end corresponding to the soma. Further efforts have been put
on models for signal propagation along the axon.  Shortly after the
publication of Hodgkin and Huxley's model for the diffusion of
electric potential in the squid giant axon, a more analytically
treatable model was proposed by FitzHugh and Nagumo; the model is able
to catch the main mathematical properties of excitation and
propagation using
\begin{itemize}
\item[$\circ$] a voltage-like variable having cubic nonlinearity that
  allows regenerative self-excitation via a positive feedback, and
\item[$\circ$] a recovery variable having a linear dynamics that
  provides a slower negative feedback.
\end{itemize}
In our model the axon has length $\ell$, i.e. the space variable $x$
in the above equations ranges in an interval $(0, \ell)$, where the
soma (the cell body) is identified with the point 0.

\medskip

There is a large evidence in the literature that realistic
neurobiological models shall incorporate stochastic terms to model
real inputs.  It is classical to model the random perturbation with a
Wiener process, compare \cite{ricciardi}, as it comes from a central
limit theorem applied to a sequence of independent random variables.

However, there is a considerable interest in literature to apply
different kind of noises: we shall mention long-range dependence
processes and self-similar processes, as their features better model
the real inputs: see the contributions in~\cite[Part II]{RanDin03}.
Further, they can be justified theoretically as they arise in the so
called Non Central Limit Theorem, see for instance \cite{DM,Ta1}.

The fractional Brownian motion is of course the most studied process
in the class of Hermite processes due to its significant importance in
modeling. It is not only selfsimilar, but also exhibits long-range
dependence, i.e., the behaviour of the process at time $t$ does depend
on the whole history up to time $t$, stationarity of the increments and
continuity of trajectories.

\subsection{The abstract formulation}
\label{sez3}

In the following, as long as we allow for variable coefficients in the
diffusion operator, we can let the edges of the neuronal network to be
described by the interval $[0,1]$.
The general form of the equation we are concerned with can be written
as a system in the space $\bX = (L^2(0,1))^2 \times \bR
\times L^2(0,1)$ for the unknowns $(u, u_d, d, v)$:
\begin{equation}
  \label{eq:system-1}
  \begin{aligned}
    \tfrac{\partial}{\partial t} u(t,x) &= \tfrac{\partial}{\partial
      x} \left(c(x) \tfrac{\partial}{\partial x} u(t,x) \right) - p(x)
    u(t,x) - v(t,x) + \theta(u(t,x)) + \tfrac{\partial}{\partial t}
    \zeta^u(t,x)
    \\
    \tfrac{\partial}{\partial t} u_d(t,x) &= \tfrac{\partial}{\partial
      x} \left(c_d(x) \tfrac{\partial}{\partial x} u_d(t,x) \right) -
    p_d(x) u_d(t,x) + \tfrac{\partial}{\partial t} \zeta^d(t,x)
    \\
    \tfrac{\partial}{\partial t} d(t) &= - \gamma d(t) - \left(c(0)
      \tfrac{\partial}{\partial x}u(t,0) - c_d(1)
      \tfrac{\partial}{\partial x}u_d(t,1) \right)
    \\
    \tfrac{\partial}{\partial t} v(t,x) &= u(t,x) - \epsilon v(t,x) +
    \tfrac{\partial}{\partial t} \zeta^v(t,x)
  \end{aligned}
\end{equation}
under the following continuity, boundary and initial conditions
\begin{equation}
  \label{eq:system-2}
  \begin{aligned}
    &d(t) = u(t,0) = u_d(t,1), &\qquad t \ge 0
    \\
    &\tfrac{\partial}{\partial x}u(t,1) = 0, \qquad
    \tfrac{\partial}{\partial x}u_d(t,0) = 0, &\qquad t \ge 0
    \\
    &u(0,x) = u_0(x), \qquad v(0,x) = v_0(x), \qquad u_d(0,x) =
    u_{d;0}(x).
  \end{aligned}
\end{equation}

Throughout the paper we shall assume that the coefficients in
(\ref{eq:system-1}) satisfy the following conditions.

\pagebreak[1]

\begin{assumpt}$\null$
  \begin{itemize}
  \item The function $\theta: \bR \to \bR$ satisfies some
    dissipativity conditions: there exists $\lambda \ge 0$ such that
    \begin{equation}
      \label{eq:intro-h}
      \begin{aligned}
        &\text{for $h(u) = -\lambda u + \theta(u)$ it holds}
        \\
        & [h(u)-h(v)](u-v) \le 0 \quad \forall\ u, v \in \bR; \quad
        |h(u)| \le c (1 + |u|^{2\rho+1}), \quad \rho \in \bN.
      \end{aligned}
    \end{equation}
  \item $c, c_d, p, p_d \in C^1([0,1])$ are continuous, positive
    functions such that, for some $C > 0$,
    \begin{align*}
      C \le c(x), c_d(x) \le \frac1C,\quad C' \le p(x)-\lambda, p_d(x)
      \le \frac{1}{C'};
    \end{align*}
  \item $\gamma > 0$, $\epsilon > 0$ are given constants.
  \end{itemize}
\end{assumpt}

\begin{rem}
  The function $\theta: \bR \to \bR$, in the classical model of
  FitzHugh, is given by $\theta(u) = u(1-u)(u-\xi)$ for some $\xi \in
  (0,1)$; it satisfies (\ref{eq:intro-h}) with $\lambda =
  \frac{1}{3}(\xi^2-\xi+1)$. Other examples of nonlinear conditions
  are known in the literature, see for instance \cite{Iz04} and the
  references therein.
\end{rem}

Our aim is to write equation \eqref{eq:system-1}, endowed with the
conditions in (\ref{eq:system-2}), in an abstract form in the Hilbert
space $\bX = (L^2(0,1))^2 \times \bR \times L^2(0,1)$. We also
introduce the Banach space $\bY = (C([0,1]))^2 \times \bR \times
L^2(0,1)$ that is continuously (but not compactly) embedded in $\bX$.
In this section we establish the basic framework that we need in order
to solve the abstract problem. To this aim we need to prove that the
linear part of the system defines a linear, unbounded operator $\bA$
that generates on $\bX$ an analytic semigroup. We
shall also study the dissipativity of $\bA$ and of the nonlinear term
$\bF$ (see (\ref{eq:def_F})).

\medskip

On the domain
\begin{equation}
  \label{eq:def-A-1}
  D(\bA):=\left\{
    \begin{aligned}
      {\mathfrak v} := (u,v,d,u_d)^{\top} \in (H^2(0,1))^2 \times
      {\bR} \times L^2(0,1) \quad \text{s.\ th.} \quad u(0) =
      u_d(1) = d, \quad&
      \\
      u'(1) = 0, \quad u'_d(0) = 0, \quad c(0) u'(0) + c_d(1) u_d'(1)
      = 0&
    \end{aligned}
  \right\}
\end{equation}
we define the operator $\bA$ by setting
\begin{equation}
  \label{eq:def-A-2}
  {\bA}{\mathfrak v} :=
  \begin{pmatrix}
    (c u')' - p u + \lambda u - v
    \\
    (c_d u_d')' - p_d u_d
    \\
    -\gamma d - \left(c(0) u'(0) - c_d(1) u_d'(1)\right)
    \\
    u - \epsilon v
  \end{pmatrix}
\end{equation}

In order to treat the nonlinearity in our system, we introduce the
Nemitsky operator $\Theta$ on $L^2(0,1)$ such that $\Theta(u)(x) =
h(u(x))$ for all $u \in C([0,1]) \subset L^2(0,1)$. Then we
define $\bF$ on $\bX$ by setting
\begin{equation}
  \label{eq:def_F}
  \begin{aligned}
      \bF({\mathfrak v}) &= (\Theta(u), 0, 0, 0)^{\top}
      \\
      \text{on the domain } D(\bF) &= \left\{ (u,v,d,u_d)^{\top} \in
      \bX \,:\, u \in C([0,1]) \right\}
  \end{aligned}
\end{equation}

\begin{rem}
  In the above setting, the function $\bF$ satisfies the conditions in
  Assumption \ref{Assumption F}.
\end{rem}

Finally, setting $B(t) = (\zeta^u(t), \zeta^v(t), 0,
\zeta^d(t))^{\top}$, we obtain that the initial value problem
associated with \eqref{eq:system-1}--\eqref{eq:system-2} can be
equivalently formulated as an abstract stochastic Cauchy problem
\begin{equation}
  \label{eq:1305-1}
  \left\{
    \begin{aligned}
      {\rm d}{\mathfrak v}(t) &= [{\bA}{\mathfrak v}(t) +
      {\bF}({\mathfrak v}(t)) \, {\rm d}t + {\rm d}B(t),\qquad t\geq
      0,
      \\
      {\mathfrak v}(0) &= {\mathfrak v}_0,
    \end{aligned}
  \right.
\end{equation}
where the initial value is given by ${\mathfrak v}_0:=(u_0, v_0,
u_0(0), u_{d;0})^{\top} \in \bX$.

In the next section we shall prove that the leading operator $\bA$ in
Eq.~(\ref{eq:1305-1}) satisfies the condition in Assumption
\ref{Assumption A}.  According to Theorem \ref{thm:ex-intro}, this
implies that there exists a unique solution to problem
(\ref{eq:1305-1}) whenever the noise $(B(t))_{t \ge 0}$ is a
fractional Brownian motion with Hurst parameter $H > \frac12$, or a
bifractional Brownian motion with $H > \frac12$ and $K \ge 1/2H$, or
an Hermite process with selfsimilarity order $H > \frac12$, or, more
generally, a process that satisfies Assumption \ref{Assumption 1}.

\begin{thm}
  The proposed model for a neuron cell, endowed with a stochastic
  input that satisfies the conditions in  Assumption \ref{Assumption
  1}, has a unique solution on the time interval $[0,T]$, for
  arbitrary $T > 0$, which belongs to
  \begin{align*}
    L^2_{\cF}(\Omega;C([0,T];\bX)) \cap L^2_{\cF}(\Omega;L^2([0,T];\bV))
  \end{align*}
  and depends continuously on the initial condition.
\end{thm}

\subsection{The well-posedness of the linear system}

As stated above, we can refer to some results in the existing
literature in order to prove well-posedness and further qualitative
properties of our system: the main references here are
\cite{CM07,MR07, Mu07}.

Our first remark is that, neglecting the recovery variable $v$, the
(linear part of the) system for the unknown $(u, u_d, d)$ is a
diffusion equation on a network with {\em dynamical boundary
  conditions}:
\begin{equation}
  \label{eq:linear-no-recovery}
  \begin{aligned}
    \tfrac{\partial}{\partial t} u(t,x) &= \tfrac{\partial}{\partial
      x} \left(\tfrac{\partial}{\partial x} c(x) u(t,x) \right) - p(x)
    u(t,x) + \lambda u(t,x)
    \\
    \tfrac{\partial}{\partial t} u_d(t,x) &= \tfrac{\partial}{\partial
      x} \left(\tfrac{\partial}{\partial x} c_d(x) u_d(t,x) \right) -
    p_d(x) u_d(t,x)
    \\
    \tfrac{\partial}{\partial t} d(t) &= - \gamma d(t) - \left(c(0)
      \tfrac{\partial}{\partial x}u(t,0) - c_d(1)
      \tfrac{\partial}{\partial x}u_d(t,1) \right)
  \end{aligned}
\end{equation}
Such systems are already present in the literature.  Let us define
$\cX = (L^2(0,1))^2 \times \bR$ and introduce the operator
\begin{equation*}
  \cA
  \begin{pmatrix}u\\ u_d\\ d\end{pmatrix}
  =
  \begin{pmatrix}(c u')'- p u + \lambda u
    \\ (c_d u_d')' - p_d u_d
    \\ - \gamma_1 d - (c(0) u'(0) - c_d(1) u_d'(1))
  \end{pmatrix}
\end{equation*}
with {\em coupled} domain
\begin{align*}
  D(\cA) = \left\{ (u, u_d, d)^{\top} \in (H^2(0,1))^2 \times \bC \,:\,
    u(0) = u_d(1) = d \right\}
\end{align*}
Then, by quoting for instance the papers \cite{MR07, Mu07}, we can
state the following result.

\begin{prop}\label{prop-cA}
  The operator $(\cA, D(\cA))$ is self-adjoint and dissipative and it
  has compact resolvent; by the spectral theorem, it generates a
  strongly continuous, analytic and compact semigroup $({\mathcal
    S}(t))_{t\geq 0}$ on the Hilbert space ${\mathcal X}$.
\end{prop}

\medskip

The next step is to introduce the operator $\bA$ on the space $\bX =
\cX \times L^2(0,1)$. We can think $\bA$ as a matrix operator in the
form
\begin{equation*}
  \bA =
  \begin{pmatrix}
    \cA & -P_{1}
    \\
    P_1^{\top} & - \epsilon
  \end{pmatrix}
\end{equation*}
where $P_1$ is the immersion on the first coordinate of $\cX$: $P_1 v
= (v, 0, 0)^\top$, while $P_1^{\top}(u,u_d,v)^\top = u$.

In order to prove the generation property of the operator $\bA$, we
introduce the Hilbert space
\begin{equation*}
  \bV := \left\{
    \begin{aligned}
      {\mathfrak v} := (u,u_d,d,v)^{\top} \in (H^1(0,1))^2 \times {\bR}
      \times L^2(0,1) \quad \text{s.\ th.} \quad&
      \\
      u(0) = u_d(1) = d&
    \end{aligned}
  \right\}
\end{equation*}
and the sesquilinear form ${\mathfrak a}: \bV \times
\bV \to {\bR}$ defined by
\begin{align*}
  {\mathfrak a}({\mathfrak u}^{(1)}, {\mathfrak u}^{(2)}) :=
  &\int_0^{1} p(x) \, u^{(1)}(x) \overline{u^{(2)}(x)} + c(x) \,
  (u^{(1)})'(x) \overline{(u^{(2)})'(x)} \, {\rm d}x
  \\
  &+ \int_0^{1} p_d(x) u_d^{(1)}(x) \overline{u_d^{(2)}(x)} + c_d(x)
  \, (u_d^{(1)})'(x) \overline{(u_d^{(2)})'(x)} \, {\rm d}x
  \\
  &+ \int_0^{1} u^{(1)}(x) \overline{v^{(2)}(x)} - v^{(1)}(x)
  \overline{u^{(2)}(x)} + \epsilon v^{(1)}(x) \overline{v^{(2)}(x)} \,
  {\rm d}x + \gamma d^{(1)} \overline{d^{(2)}}.
\end{align*}

\begin{prop}
  \label{prop-bA}
  The operator $\bA$ generates a strongly continuous, analytic
  semigroup $({S}(t))_{t\geq 0}$ on the Hilbert space $\bX$ that is
  uniformly exponentially stable: there exist $M\ge 1$ and $\omega >
  0$ such that $\|S(t)\|_{L(\bX)} \le M e^{-\omega t}$ for all $t \ge
  0$.
\end{prop}

\begin{proof}
  We first notice that $(\bA, D(\bA))$ is the operator
  associated with the form $({\mathfrak a}, \bV)$: compare for
  instance \cite[Lemma 4.2]{CM07}.

  The form $({\mathfrak a}, \bV)$ is non-symmetric, as it can be seen
  by setting ${\mathfrak u}^{(1)} = (1,1,1,0)^\top$ and ${\mathfrak
    u}^{(2)} = (1,1,1,1)^\top$ and computing ${\mathfrak a}({\mathfrak
    u}^{(1)}, {\mathfrak u}^{(2)}) - {\mathfrak a}({\mathfrak
    u}^{(2)}, {\mathfrak u}^{(1)}) = 2$. However $({\mathfrak a},
  \bV)$ is densely defined, coercive and continuous, see \cite[Theorem
  4.3]{CM07}.  Then, the properties of the semigroup follow from
  standard theory of Dirichlet forms, compare \cite{Ou05}.
\end{proof}

Notice that the operator $\bA$ is not self-adjoint, as the
corresponding form ${\mathfrak a}$ is not symmetric; also, since $\bV$
is not compactly embedded in $\bX$, it is easily seen that the
semigroup generated by $\bA$ is not compact hence it is not
Hilbert-Schmidt. For our purposes, they are of fundamental importance the
following observations.

\begin{rem}
  \label{rem:V=D(A^1/2)}
  The form domain $\bV$ is isometric to the fractional domain power
  $D((-A)^{1/2})$. This follows since the numerical range of the form
  ${\mathfrak a}$ is contained in a parabola, compare \cite[Corollary
  6.2]{CM07}, and then by an application of a known result of
  McIntosh \cite[Theorems A and C]{McI82}.
\end{rem}

\begin{rem}
  \label{rem:coercive}
  The form ${\mathfrak a}$ is real-valued and coercive, hence
  \begin{align*}
    \langle -\bA u, u \rangle = {\mathfrak a}(u,u) \ge \omega
    \|u\|^2_{\bV}
  \end{align*}
  for some $\omega > 0$.
\end{rem}

Although we shall not use directly the next result in
this paper, we can characterize further the specturm of $\bA$ in the
complex plane. This result was first investigated in \cite{CM06}; we
provide here our proof for the convenience of the reader.

\begin{lem}
  The spectrum of $\bA$ in the complex plane is contained in the union
  of the (discrete, real and negative) spectrum of $\cA$ and a bounded
  $B$.
\end{lem}

\begin{proof}
  To compute the spectrum we apply \cite[Theorem 2.4]{Na89}. There it
  is proved that for any $\lambda \not\in \sigma(\cA) \cup \{-\epsilon\}$
  it holds $\lambda \in \sigma(\bA)$  if and only if $0 \in
  \sigma(\Delta_\lambda(\cA))$, where $\Delta_\lambda(\cA)$ is the operator
  \begin{align*}
    (\lambda - \cA) + \frac{1}{\epsilon + \lambda}P_1 P_1^{\top}.
  \end{align*}
  By standard results on additive bounded perturbations of operators,
  we notice that
  \begin{align*}
    \{\lambda\,:\, 0 \in \sigma(\Delta_\lambda(\cA))\}^{c} \supset
    \{\lambda\,:\, \left\|R(\lambda,\cA) \frac{1}{\lambda + \epsilon}
      P_1 P_1^\top \right\| < 1\} \supset \{\lambda\,:\,
    \frac{1}{|\lambda + \omega|} \frac{1}{|\lambda + \epsilon|} < 1\}
  \end{align*}
  where $-\omega = s(\cA)$ is the {\em spectral bound} of $\cA$, that
  is a negative real number by Proposition \ref{prop-cA}. Therefore,
  setting $B = \{\lambda\,:\, |\lambda + \omega| \, |\lambda +
  \epsilon| < 1\}$ we have that $B$ is a bounded subset of the complex
  plane
  .
\end{proof}

\ack Part of this paper has been written while the second author was
visiting the research center ``Centro Internazionale per la Ricerca
Matematica (C.I.R.M.)'' and the University of Trento in May 2009. He
warmly acknowledges support and hospitality.

\end{document}